\documentclass[a4paper,twoside]{article}
\usepackage[utf8]{inputenc}
\usepackage{tikz}
\usepackage{tikz-cd}
\usepackage[hidelinks]{hyperref}
\usetikzlibrary{shapes,arrows}
\tikzstyle{block}=[draw opacity=0.7,line width=1.4cm]
\usepackage{amssymb}
\usepackage{amsmath}
\usepackage{multirow}
\usepackage{amsthm}
\usepackage{amsfonts}
\usepackage[mathscr]{eucal}
\usepackage{graphicx}
\usepackage{hyphenat}
\usepackage{mathtools}
\makeatletter

\def\dom{\mathop{\mathrm{Dom}}\nolimits}

\def\Str{\mathop{\mathrm{Str}}\nolimits}

\def\str#1{\mathbf {#1}}
\def\arity#1{a(\rel{}{#1})}

\def\rel#1#2{R_{\mathbf{#1}}^{#2}}
\def\func#1#2{F_{\mathbf{#1}}^{#2}}

\def\K{{\mathcal K}}

\def\Fraisse{Fra\"{\i}ss\' e}
\theoremstyle{definition}
\newtheorem{defn}{Definition}[section]

\newtheorem*{remark}{Remark}
\theoremstyle{remark}

\theoremstyle{plain}
\newtheorem{thm}{Theorem}[section]

\begin{document}
\bibliographystyle{plain}

\title{Ramsey Classes with Closure Operations\\
{\Large (Selected Combinatorial Applications)}}

\author{
   \normalsize {\bf Jan Hubi\v cka$^*$}\\
   \normalsize {\bf Jaroslav Ne\v set\v ril}%
\thanks {The Computer Science Institute of Charles University (IUUK) is supported by grant ERC-CZ LL-1201 of the Czech Ministry of Education and CE-ITI P202/12/G061 of GA\v CR.}
\\
{\small Computer Science Institute of Charles University (IUUK)}\\
   {\small Charles University}\\
   {\small Malostransk\' e n\' am. 25, 11800 Praha, Czech Republic}\\
   {\normalsize 118 00 Praha 1}\\
   {\normalsize Czech Republic}\\
   {\small \{hubicka,nesetril\}@iuuk.mff.cuni.cz}
}

\date{ Dedicated to old friend Ron Graham}
\maketitle
\begin{abstract}

We state the Ramsey property of classes of ordered structures with closures and
given local properties.  This generalises many old and new results: the Ne\v set\v
ril-R\"odl Theorem,  the authors Ramsey lift of bowtie-free graphs as well as the Ramsey Theorem for Finite Models
(i.e.  structures with both functions and relations) thus providing the ultimate
generalisation of Structural Ramsey Theorem. 
We give here a more concise reformulation of the recent paper ``All those Ramsey classes (Ramsey classes with closures and forbidden homomorphisms)'' and the main purpose of this paper is to show several applications. Particularly we prove the Ramsey property of ordered sets with equivalences on the power set,  Ramsey theorem for Steiner systems,  Ramsey theorem for resolvable designs and a partial Ramsey type results for $H$-factorizable graphs.
All of these results are natural, easy to state, yet proofs
involve most of the theory developed.
\end{abstract}

\maketitle
\section {Introduction}
Extending classical early Ramsey-type results~\cite{Graham1971,Graham1972,Nevsetvril1976,Halpern1966,Abramson1978,Milliken1979}, the structural Ramsey theory originated at
the beginning of 1970's, see~\cite{Nevsetvril1996} for references  of the early history of the subject.

Let us start with the key definition of this paper.  Let $\mathcal K$ be a class of
structures endowed with embeddings. For objects $\str{A},\str{B}\in \mathcal K$
denote by ${\str{B}\choose \str{A}}$ the set of all sub-objects
$\widetilde{\str{A}}$ of $\str{B}$ which are   isomorphic
to $\str{A}$. (By a sub-object we mean that the inclusion is an embedding.)
Using this notation the central definition of this paper gets the following form:
A class $\mathcal K$ is a {\em Ramsey class} if for every two objects
$\str{A}$ and $\str{B}$ and for every positive integer $k$ there exists
object $\str{C} \in \mathcal K$ such that the following holds: For every partition
${\str{B}\choose \str{A}}$ in $k$ classes there exists
$\widetilde{\str B} \in {\str{C}\choose \str{B}}$ such that
${\widetilde{\str{B}}\choose \str{A}}$ belongs to one class of the
partition.  It is usual to shorten the last part of the definition as
$\str{C} \longrightarrow (\str{B})^{\str{A}}_k$.

In \cite{Graham1971,Graham1972,Nevsetvril1976,Nevsetvril1996,Halpern1966,Abramson1978,Milliken1979} are given many examples of Ramsey classes and the recent interest of people in topological dynamics, model theory \cite{Kechris2005,The2013,Solecki2016} and, of course, combinatorists 
led to a hunt for more and more Ramsey classes. Particularly the following theorem has been recently proved in \cite{Hubicka2016} (see also a strengthening in~\cite{Evans3}):

\begin{thm}[Ramsey Theorem for Finite Models]
\label{thm:models}
For every language $L$ involving both relations and functions 
the class of all linearly ordered $L$ structures is a Ramsey class.
\end{thm}

Below in Section \ref{intro} we formulate this result  more precisely as Theorem \ref{thm:models2} after introducing 
all relevant notions. But already at this place let us add that  Theorem \ref{thm:models}
nicely complements results for relational structures  (Abramson-Harrington~\cite{Abramson1978}, Ne\v set\v ril-R\"odl~\cite{Nevsetvril1977b}),
who proved this result for finite relational systems. It is exactly the presence of functions (and function symbols) which makes Theorem \ref{thm:models} interesting and which presents some challenging problems. In turn  this aspect has some interesting applications. This we want to demonstrate here by several combinatorial examples: equivalences on the power set (thus solving a problem of \cite{Ivanov2015}), Steiner systems (proving \cite{bhat2016ramsey} in a different way), Ramsey theorem for resolvable designs, degenerate graphs and more generally $\mathcal{H}$-decomposable graphs.
As we shall demonstrate all these structures form 
Ramsey classes. These all seemingly very special theorems can be uniformly treated by means of the main
 result of \cite{Hubicka2016}.
Although the nature of \cite{Hubicka2016} is a blend of combinatorics and model theory we concentrate here on combinatorial side of the subject.
This paper is self-contained and in the next section we shall explain all the relevant notions.

\section {Preliminaries}
\label{intro}

We work with the following structures involving relations and symmetric partial functions.

Let $L=L_{\mathcal{R}}\cup L_{\mathcal{F}}$ be a language involving relational symbols $\rel{}{}\in L_{\mathcal{R}}$ and function symbols $F\in L_{\mathcal{F}}$ each having associated arities denoted by $\arity{}>0$ for relations and $d(F)>0, r(F)>0$ for functions. 
An \emph{$L$-model}, or sometimes \emph{$L$-structure}, is a structure $\str{A}$ with {\em vertex set} $A$, functions 
$\func{A}{}:\dom(\func{A}{})\to {A\choose {r(\func{}{})}}$,
 $\dom(\func{A}{})\subseteq A^{d(\func{}{})}$ for $\func{}{}\in L_F$ and relations $\rel{A}{}\subseteq A^{\arity{}}$ for $\rel{}{}\in L_R$. (Note that  by ${A\choose {r(\func{}{})}}$ we denote, as it is usual in this context, the set of all $r(\func{}{})$-element subsets of $A$.)
Set $\dom(\func{A}{})$ is called the {\em domain} of function $\func{}{}$  in $\str{A}$.

Note also that we have chosen to have the range of the function symbols to be the set of subsets (not  tuples). This is motivated by \cite{Evans3} where we deal with (Hrushovski) extension properties and we need a ``symmetric'' range. However from the point of view of Ramsey theory this is not an important issue.
 
The language is usually fixed and understood from the context (and it is in most cases denoted by $L$).  If set $A$ is finite we call $\str A$ a \emph{finite $L$-model} or \emph{$L$-structure}. We consider only structures with countably many vertices. 
If language $L$ contains no function symbols, we call $L$ a {\em relational language} and an $L$-structure is also called a \emph{relational $L$-structure}.
Every function symbol $\func{}{}$ such that $d(\func{}{})=1$ is a {\em unary function}. Unary relation is of course just defining a subset of elements of $A$.

A \emph{homomorphism} $f:\str{A}\to \str{B}$ is a mapping $f:A\to B$ satisfying for every $\rel{}{}\in L_{\mathcal R}$ and for every $\func{}{}\in L_\mathcal F$ the following two statements:
\begin{enumerate}
\item[(a)] $(x_1,x_2,\ldots, x_{\arity{}})\in \rel{A}{}\implies (f(x_1),f(x_2),\ldots,f(x_{\arity{}}))\in \rel{B}{}$, and,
\item[(b)]  $f(\dom(\func{A}{}))\subseteq \dom(\func{B}{})$
 and $f(\func{A}{}(x_1,x_2,\allowbreak \ldots, x_{d(\func{}{})}))=\func{B}{}(f(x_1),f(x_2),\allowbreak \ldots,\allowbreak f(x_{d(\func{}{})}))$ for every $(x_1,x_2,\allowbreak \ldots, x_{d(\func{}{})})\in \dom(\func{A}{})$.
\end{enumerate} 
For a subset $A'\subseteq A$ we denote by $f(A')$ the set $\{f(x);x\in A'\}$ and by $f(\str{A})$ the homomorphic image of a structure.

 If $f$ is injective, then $f$ is called a \emph{monomorphism}. A monomorphism is called an \emph{embedding} if
for every $\rel{}{}\in L_{\mathcal R}$ and $\func{}{}\in L_\mathcal F$ the following holds:
\begin{enumerate}
\item[(a)] $(x_1,x_2,\ldots, x_{\arity{}})\in \rel{A}{}\iff (f(x_1),f(x_2),\ldots,f(x_{\arity{}}))\in \rel{B}{}$, and,
\item[(b)]  
$(x_1,x_2,\ldots, x_{d(\func{}{})})\in\dom(\func{A}{}) \iff (f(x_1),\ldots,f(x_{d(\func{}{})}))\in \dom(\func{B}{}).$
\end{enumerate}
 
  If $f$ is an embedding which is an inclusion then $\str{A}$ is a \emph{substructure} (or \emph{subobject}) of $\str{B}$. For an embedding $f:\str{A}\to \str{B}$ we say that $\str{A}$ is \emph{isomorphic} to $f(\str{A})$ and $f(\str{A})$ is also called a \emph{copy} of $\str{A}$ in $\str{B}$. Thus $\str{B}\choose \str{A}$ is defined as the set of all copies of $\str{A}$ in $\str{B}$. Finally, $\Str(L)$ denotes the class of all finite $L$-models and all their embeddings.

For $L$ containing a binary relation $\rel{}{\leq}$ we denote by $\overrightarrow{\Str}(L)$ the class of all finite models (or $L$-structures) $\str{A}\in\Str(L)$ where the set $A$ is linearly ordered by the relation $\rel{}{\leq}$. $\overrightarrow{\Str}(L)$ is of course considered with all monotone (i.e. order preserving) embeddings.

The Ramsey theorem of finite models can now be stated as follows (\cite{Evans3}):

\begin{thm}
\label{thm:models2}
For every language $L$ involving both relations and functions and containing a binary relation $\rel{}{\leq}$ the class of all ordered $L$-structures $\overrightarrow{\Str}(L)$ is a Ramsey class.
\end{thm}
Ramsey Theorem for Finite Models
(i.e.  for structures with both functions and relations) thus provides the ultimate
generalisation of  structural  Ramsey Theorem for relational structures \cite{Abramson1978,Nevsetvril1976}. In this paper this theorem will be our starting point.
We remark that a generalisation of Ne\v set\v ril-R\"odl Theorem for languages involving both functions and relations is also a main result of~\cite{Solecki2012}. Our definition of
functions however differs from the one used in~\cite{Solecki2012}.

Before stating more special (and stronger) Ramsey type statements we need to review 
some more  model-theoretic notions (see
e.g.~\cite{Hodges1993}).

\begin{figure}
\centering
\includegraphics{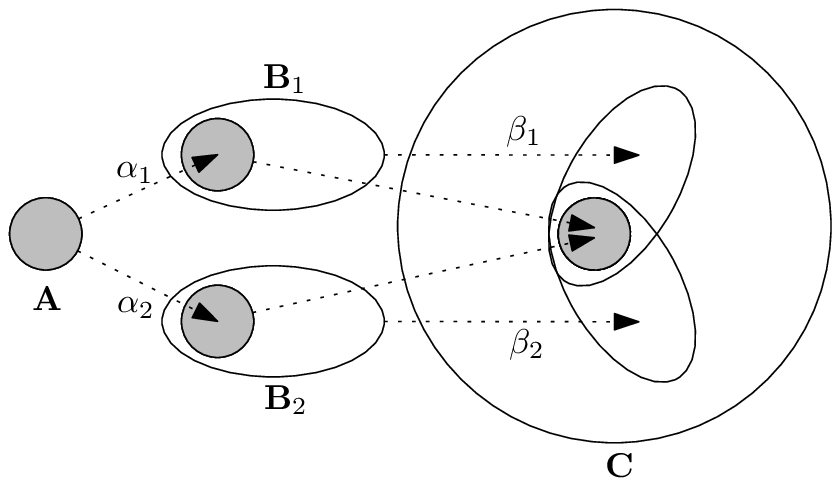}
\caption{An amalgamation of $\str{B}_1$ and $\str{B}_2$ over $\str{A}$.}
\label{amalgamfig}
\end{figure}
Let $\str{A}$, $\str{B}_1$ and $\str{B}_2$ be relational structures and $\alpha_1$ an embedding of $\str{A}$
into $\str{B}_1$, $\alpha_2$ an embedding of $\str{A}$ into $\str{B}_2$, then
every structure $\str{C}$
 with embeddings $\beta_1:\str{B}_1 \to \str{C}$ and
$\beta_2:\str{B}_2\to\str{C}$ such that $\beta_1\circ\alpha_1 =
\beta_2\circ\alpha_2$ is called an \emph{amalgamation} of $\str{B}_1$ and $\str{B}_2$ over $\str{A}$ with respect to $\alpha_1$ and $\alpha_2$. See Figure~\ref{amalgamfig}.
We will call $\str{C}$ simply an \emph{amalgamation} of $\str{B}_1$ and $\str{B}_2$ over $\str{A}$
(as in the most cases $\alpha_1$ and $\alpha_2$ can be chosen to be inclusion embeddings).

We say that an amalgamation is \emph{strong} when $\beta_1(x_1)=\beta_2(x_2)$ 
only if $x_1\in \alpha_1(A)$ and $x_2\in \alpha_2(A)$.  Less formally, a strong
amalgamation glues together $\str{B}_1$ and $\str{B}_2$ with the overlap no
greater than the copy of $\str{A}$ itself.  A strong amalgamation is \emph{free} if there are no tuples in any relations of $\str{C}$ and no tuples in $\dom(\func{C}{})$, $\func{}{}\in L$,  using vertices of both
$\beta_1(B_1\setminus \alpha_1(A))$ and $\beta_2(B_2\setminus \alpha_2(A))$.

An \emph{amalgamation class} is a class $\K$ of finite structures satisfying the following three conditions:
\begin{enumerate}
\item {Hereditary property:} For every $\str{A}\in \K$ and a substructure $\str{B}$ of $\str{A}$ we have $\str{B}\in \K$;
\item {Joint embedding property:} For every $\str{A}, \str{B}\in \K$ there exists $\str{C}\in \K$ such that $\str{C}$ contains both $\str{A}$ and $\str{B}$ as substructures;
\item {Amalgamation property:} 
For $\str{A},\str{B}_1,\str{B}_2\in \K$ and $\alpha_1$ embedding of $\str{A}$ into $\str{B}_1$, $\alpha_2$ embedding of $\str{A}$ into $\str{B}_2$, there is $\str{C}\in \K$ which is an amalgamation of $\str{B}_1$ and $\str{B}_2$ over $\str{A}$ with respect to $\alpha_1$ and $\alpha_2$.
\end{enumerate}

\section{Previous work --- multiamalgamation} 

We now refine amalgamation classes.
Our aim is to describe  even stronger sufficient criteria for  Ramsey classes.

\begin{figure}
\centering
\includegraphics{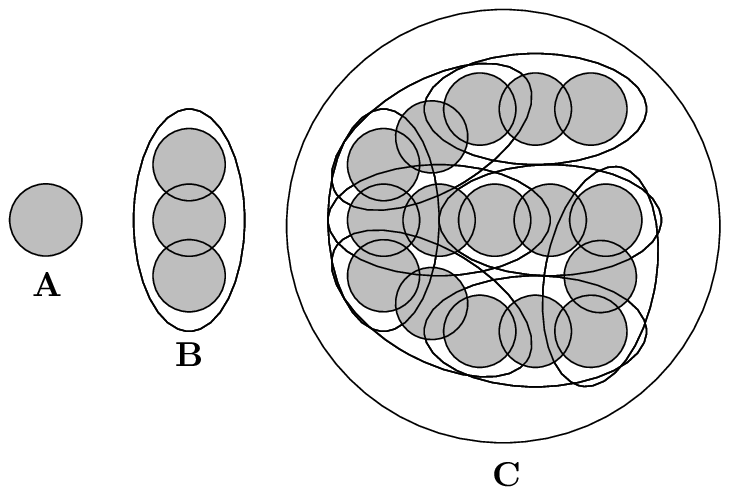}
\caption{Construction of a Ramsey object by multiamalgamation.}
\label{fig:multiamalgam}
\end{figure}
First, we develop a generalised notion of amalgamation which will serve as a useful tool for the construction of Ramsey objects. As schematically depicted in Figure~\ref{fig:multiamalgam}, Ramsey objects are a result of
amalgamation of multiple copies of a given structure which are all performed at once. In a non-trivial class this leads to many problems. Instead of working with complicated amalgamation diagrams we  split the amalgamation into two steps---the construction of (up to
isomorphism unique) free amalgamation (which yields an incomplete or ``partial'' structure) followed then by a completion. Formally this is done as follows:

\begin{defn}
\label{def:irreducible}
An $L$-structure $\str{A}$ is \emph{irreducible} if 
$\str{A}$ is not a free amalgam of two proper substructures of $\str{A}$.
\end{defn}

\begin{remark}
In the case of relational structures this definition allows a combinatorial description: $\str{A}$ is irreducible if
for every pair of distinct vertices $u$, $v$ there is  a tuple $\vec{t}\in \rel{A}{}$
(of some relation $\rel{}{}\in L_{\mathcal{R}}$) such that $\vec{t}$ contains both $u$ and $v$. In the case of structures with function symbols this is more complicated as  we have to consider closures defined below. 
\end{remark}

Thus the irreducibility is meant with respect to the free amalgamation. The irreducible structures are our building blocks.
Moreover in structural Ramsey theory we are fortunate that most structures are (or may be interpreted as) irreducible. And in the most interesting case, the structures may be completed to irreducible structures. 
This will be introduced now by means of the following variant of the homomorphism notion.
\begin{defn}
 A homomorphism
$f:\str{A}\to\str{B}$ is \emph{homomorphism-embedding}  if $f$ restricted to any irreducible substructure of $\str{A}$      
is an embedding to $\str{B}$.
\end{defn}
While for (undirected) graphs the homomorphism and homomorphism-em\-bed\-ding coincide, for structures they differ.
For example any homomorphism-embedding of the Fano plane into a  hypergraph is actually an embedding.

\begin{defn}
\label{defn:completion}
Let $\str{C}$ be a structure. An irreducible structure $\str{C}'$ is a \emph{completion}
of $\str{C}$ if there exists a homomorphism-embedding $\str{C}\to\str{C}'$.

Let $\str{B}$ be an irreducible substructure of $\str{C}$.  We
say that irreducible structure $\str{C}'$ is a \emph{completion of $\str{C}$ with respect to
copies of $\str{B}$} if there exists a function
$f:C\to C'$ such that for every  $\widetilde{\str{B}}\in
{\str{C}\choose \str{B}}$ the function $f$ restricted to $\widetilde{B}$
is an embedding of $\widetilde{\str{B}}$ to $\str{C}'$.

If $\str{C}'$ belongs to a given class $\K$, then $\str{C}'$ is called a \emph{$\K$-completion of $\str{C}$ with respect to copies of $\str{B}$}.
\end{defn}
\begin{remark}[on completion and holes]
Completion may be seen as a generalised form of amalgamation.
To see that let $\K$ be a class of irreducible structures. The 
amalgamation property of $\K$ can be equivalently formulated as follows: For $\str{A}$, $\str{B}_1$, $\str{B}_2 \in \K$ and $\alpha_1$ embedding of $\str{A}$ into $\str{B}_1$, $\alpha_2$ embedding of $\str{A}$ into $\str{B}_2$, there is $\str{C}\in \K$ which is a  completion
of the free amalgamation (which itself is not necessarily in $\K$) of $\str{B}_1$ and $\str{B}_2$ over $\str{A}$ with respect to $\alpha_1$ and $\alpha_2$.

Free amalgamation may result in a reducible structure. The pairs of vertices  where
one vertex belong to $\str{B}_1\setminus \alpha_1(\str{A})$ and the other to $\str{B}_2\setminus \alpha_2(\str{A})$ are never both contained in a single tuple
of any relation. Such pairs can be thought of as  \emph{holes} and a completion is then a process of filling in the
holes to obtain irreducible structures (in a given class $\mathcal{K}$) while preserving all embeddings of irreducible structures.

Completion with respect to copies of $\str{B}$ is the weakest notion of completion which preserve the Ramsey property for given structures $\str{A}$ and $\str{B}$.
Note that in this case $f$ does not need to be a homomorphism-embedding (and even homomorphism).
\end{remark}

We now state all necessary conditions for our main result stated in this section:
\begin{defn}
\label{def:multiamalgamation}
Let $L$ be a language, $\mathcal R$ be a Ramsey class of finite irreducible $L$-structures.
We say that a subclass $\mathcal K$ of $\mathcal R$  is an \emph{$\mathcal R$-multiamalgamation class} if
the following conditions are satisfied:
\begin{enumerate}
 \item\label{cond:hereditary} {\em Hereditary property:} For every $\str{A}\in \K$ and a substructure $\str{B}$ of $\str{A}$ we have $\str{B}\in \K$.
 \item\label{cond:amalgamation} {\em Strong amalgamation property:}
For $\str{A},\str{B}_1,\str{B}_2\in \K$ and $\alpha_1$ embedding of $\str{A}$ into $\str{B}_1$, $\alpha_2$ embedding of $\str{A}$ into $\str{B}_2$, there is $\str{C}\in \K$ which contains a strong amalgamation of $\str{B}_1$ and $\str{B}_2$ over $\str{A}$ with respect to $\alpha_1$ and $\alpha_2$ as a substructure.
 \item\label{cond:completion} {\em Locally finite completion property:} Let $\str{B}\in \K$ and $\str{C}_0\in \mathcal R$. Then there exists $n=n(\str{B},\str{C}_0)$ such that if closed $L$-structure $\str{C}$ satisfies the following:
\begin{enumerate}
 \item there is a homomorphism-embedding from $\str{C}$ to $\str{C}_0$,
 (in other words, $\str{C}_0$ is a completion of $\str{C}$), and,
 \item every substructure of $\str{C}$ with at most $n$ vertices has a $\K$-completion.
\end{enumerate}
Then there exists $\str{C}'\in \K$ that is a completion of $\str{C}$ with respect to copies of $\str{B}$.
\end{enumerate}
\end{defn}

We can now state the main result of \cite{Hubicka2016} as:
\begin{thm}[\cite{Hubicka2016}]
\label{thm:mainstrongclosures}
Every $\mathcal R$-multiamalgamation class $\K$ is Ramsey.
\end{thm}

The proof of this result is not easy and involves interplay of several key constructions of 
structural Ramsey theory particularly Partite Lemma and Partite Construction (see \cite{Hubicka2016} for  details). 

Paper \cite{Hubicka2016} also contains several applications of this result. Here we add some new combinatorial ones in three areas mentioned in the introduction.
All our  results share a common pattern: We start with a carefully chosen special case of Theorem \ref{thm:models2} and check that all defining restrictions fit into Theorem \ref{thm:mainstrongclosures}

We find it convenient (and intuitive) also to use the terminology of closed sets.
We say that a structure $\str{A}$ is \emph{closed in $\str{B}$} if there is an embedding  $\str{A} \to \str{B}$.
So all our Ramsey theorems deal with closed sets.
\begin{remark}
For amalgamation classes of irreducible ordered structures our definition of closure is
equivalent with the model-theoretic definition of the algebraic
closure  considered in the \Fraisse{} limit of the class (see e.g. \cite{Hodges1993}).  This follows from the fact that the
closure relations are definable in the structure. Note that in \cite{Hubicka2016} the closures are handled equivalently by means of ``closure descriptions''. The advantage of such a definition is that it involves the degree condition
which  is easier to control than the abstract closure one. Here we proceed differently as our starting point is Theorem \ref{thm:mainstrongclosures}.
\end{remark}
\section{Power Set Equivalences}
\label{sec:partitions}

In this section we consider  objects which consist of sets on which there are partitions of elements, pairs or subsets of larger size. 

The essential part of more complex Ramsey classes is handling structures with equivalences defined on vertices (and even tuples of vertices).
Such an equivalence may be present latently (as for example in $S$-metric spaces with jump numbers~\cite{Sauer2013}, antipodal metric spaces~\cite{Aranda2017} or bowtie-free graphs~\cite{Hubivcka2014,Cherlin1999,Komjath1999}). It is an important fact that such equivalences may have unboundedly many equivalence classes
and thus one cannot assign labels to them and use Theorem~\ref{thm:mainstrongclosures} directly.
Here we show how to handle this.

We start our analysis with the following  definition:

\begin{defn}
A $k$-equivalence  $\str{A}$ consists of a linearly ordered set $(A, \leq_{\str{A}})$
 together with an equivalence $E_{\str{A}}$ on the set ${A}\choose {k}$. Given two partition structures $\str{A}$ and $\str{B}$ a mapping $f: A \to B$ is said to be an \emph{embedding} if $f$ is a monotone injection and if for every two sets $M, M'\in {A\choose k}$ holds
$$ (M,M') \in E_{\str{A}}  \hbox{ if and only if } (f(M),f(M')) \in E_{\str{B}}.$$

 Note that $1$-equivalence is just an equivalence on set $A$. Already $2$-equiv\-a\-lences are interesting: they may be defined as edges vs. non-edges of a graph.
We denote by $\mathcal{EQ}_k$ the class of all finite $k$-equivalences and all embeddings between them.
\end{defn}

It is easy to see that  $\mathcal{EQ}_k$  fails to be a Ramsey class. To see this it suffices to consider  already  equivalences defined by sets of vertices (i.e. equivalences on singletons only).
This corresponds to colourings of disjoint copies of complete graphs.
(In this particular case ($k=1$) one can easily save being Ramsey by a choice of special (admissible or convex) orderings).

 In fact for our purposes (i.e. application of Theorem \ref{thm:mainstrongclosures}) the equivalences with unboundedly many classes have to be interpreted so they can be viewed as structures
over a finite language. This is a place where functions (and closure operations) may be used effectively.

This
can be done as follows:

\begin{defn}
A \emph{labelled $k$-equivalence} $\str{A}$ consists of a linearly ordered set $(A, \leq_{\str{A}})$ together with 
a mapping $\ell_{\str{A}}: {{A}\choose {k}} \to \mathbb N$  (``labelling'' is thus symbolised by  $\ell$). We tacitly assume that the sets $A$ and $\mathbb N$ are disjoint.
By $\mathrm{Rg}(\ell_{\str{A}})$ we denote the range of labelling $\ell_{\str{A}}$.

Given two labelled $k$-equivalences $\str{A}$ and $\str{B}$ a mapping $f: A \to B$ is said to be an \emph{embedding} if $f$ is a monotone injection and if for every   subset $M \in {{A}\choose {k}}$ holds
$\iota(\ell_{\str{A}}(M)) =  \ell_{\str{B}}(f(M)),$
where $\iota: \mathrm{Rg}(\ell_{\str{A}}) \to  \mathrm{Rg}(\ell_{\str{B}})$ (which is defined by the above formula) is monotone and injective.

We denote by $\mathcal{LEQ}_k$ the class of all finite labelled $k$-equivalences and all embeddings between them.
\end{defn}

Thus the equality of labels and their order is also preserved. Subsets with the same label represent an equivalence class. In \cite{Hubicka2016} labelled partition structures are called \emph{pointed equivalences}.)
Explicitely, the correspondence with equivalences is as follows:

For an equivalence  $E$  on set ${{A}\choose{k}}$ we assign to every equivalence class $C$ of $E$ a label $i_C \in \mathbb N$  and a mapping $\ell_E:{{A}\choose{k}} \to \mathbb N$ which maps every element of $C$ to $i_C$ (that is, $\ell_E(M)=i_C$ for every $M\in C$).
What we obtain is a structure $\str{A}(E)$ in the language $L = \{\rel{}{},\leq,F\}$ consisting of a function symbol $\func{}{}$ (corresponding to labelling $\ell$), binary relational symbol $\leq$ and unary relational symbol $\rel{}{}$ (representing the labels).
The class, denoted for the moment by $\mathcal R=\overrightarrow{\Str}(L)$, of all ordered $L$-models is a Ramsey class by virtue of Theorem \ref{thm:models2}. The class $\mathcal{LEQ}_k$ is described by local conditions in $\mathcal R$. More precisely $\str{A}\in \mathcal R$ is in $\mathcal{LEQ}_k$ if and only if it satisfies:
\begin{enumerate}
 \item $x \notin \vec{t}\in \dom(\func{A}{})$ whenever $(x)\in \rel{}{}$;
 \item every $\vec{t}\in Dom(\func{A}{})$ consists of distinct elements of $A$;
 \item $(y)\in \rel{}{}$ whenever $\func{A}{}(\vec{t}) = \{y\}$;
 \item if $\func{A}{}(\vec{t})$ is defined and $k$-tuple $\vec{t}_2$ is created by reordering vertices of $\vec{t}$, then $\func{A}{}(\vec{t})=\func{A}{}(\vec{t}_2)$. 
\end{enumerate}
This  then axiomatises  objects of $\mathcal{LEQ}_k$ within $\mathcal R$.
Applying Theorem \ref{thm:mainstrongclosures}  we now have the following

\begin{thm}
\label{equivalence}
The class $\mathcal{LEQ}_k$ is a Ramsey class.
\end{thm}
\begin{proof}
We verify individual conditions of Definition~\ref{def:multiamalgamation} and
show that $\mathcal{LEQ}_k$ is a $\mathcal R$-multiamalgamation class.  $\mathcal{LEQ}_k$
is clearly hereditary and closed for amalgamation (which is free for $\func{}{}$ and where $\leq_\str{A}$ is completed to a linear order). The only non-trivial condition
is the completion property.

Given $\str{A},\str{B}\in \mathcal{LEQ}_k$ and $\str{C}_0\in \mathcal R$ put $n(\str{B},\str{C}_0)=0$.
For $\str{C}\longrightarrow (\str{B})^\str{A}_2$ with a homomorphism-embedding to $\str{C}_0$.
construct $\str{C}'$ on vertex set $C$ by completing $\leq_{\str{C}}$ to a linear
order $\leq_{\str{C}'}$ (by an existence of homomorphism-embedding from $\str{C}$ to ordered $\str{C}_0$ this is always possible because $\leq_\str{C}$ must be acyclic) and define $\func{C'}{}(\vec{t})=\{v\}$ if and only if  $\func{C}{}(\vec t)=\{v\}$
and  all vertices of $\vec{t}$ as well as $v$ are part of one copy of $\str{B}$ in $\str{C}$.
It is easy to check $\str{C}'$ is ordered $L$-model (and thus $\str{C}'\in \mathcal R$), all three axioms of $\mathcal{LEQ}_k$ are satisfied (and thus $\str{C}'\in \mathcal{LEQ}_k$), all
copies of $\str{B}$ in $\str{C}$ are preserved and consequently the identity is $\K$-completion of $\str{C}$ to $\str{C}'\in \mathcal{LEQ}_k$ with respect to copies of $\str{B}$.
\end{proof}

\begin{remark}
Theorem \ref{equivalence} is a strong result. For example already for 
finitely many labels it is close to the Ramsey theorem for relational structures \cite{Nevsetvril1976}: The labelling function distinguishes between edges and non-edges. For relational structures with more 
relations we can always, for each arity,  replace the overlap structure by disjoint types of edges thus forming an equivalence.

On the other hand the labelling may be interpreted as a set of imaginaries which is in line with some standard model theoretic constructions,  see \cite{Hubicka2016}.

In \cite{Ivanov2015} it is proved that the automorphism group of the \Fraisse{} limit of the class $\mathcal{EQ}_k$ is amenable group (for every $k$). The proof is via (Hrushovski) extension property for partial automorphisms (EPPA). It is suggested in \cite{Ivanov2015} that a related group   
does not contain an $\omega$-categorical extremely amenable subgroup. This is not the case. This fact is a consequence of our Theorem \ref{equivalence} using Kechris, Pestov, Todor\v cevi\' c correspondence \cite{Kechris2005,The2013}. (Another examples of  $\omega$-categorical groups not containing $\omega$-categorical amenable or extremely amenable subgroup are given in \cite{Hubicka2016}.)
\end{remark}

\section{Steiner Systems}
\label{sec:steiner}
For fixed integers $k > t \geq 2$ a \emph{partial $(k,t)$-Steiner system}
is a $k$-uniform hypergraph $G = (V,E)$ which satisfies that every $t$-element subset of $V$ is contained in at most one edge $e \in E$.
Denote by ${\mathcal S}_{k,t}$ the class of all partial $(k,t)$-Steiner systems and, as usual, 
$\vec{\mathcal{S}}_{k,t}$ the class of linearly ordered partial $(k,t)$-Steiner systems. The class $\mathcal{S}_{k,t}$ is considered with embeddings  and the class 
$\vec{\mathcal{S}}_{k,t}$ with monotone embeddings.

It is easy to see that neither $\mathcal{S}_{k,t}$ nor $\vec{\mathcal{{S}}}_{k,t}$ is a Ramsey class.
For example we can distinguish pairs of vertices by the fact whether they belong to a hyperedge or not. However this is basically the only obstacle as we can
show by an application of Theorem \ref{thm:mainstrongclosures}.

Consider the language $L=\{\leq,\func{}{}\}$ containing one binary relational symbol $\leq$ and one function symbol $\func{}{}$ of domain arity $t$ and range arity $k$.
Let $\mathcal R$ be the class of all finite ordered $L$-structures. $\mathcal R$ is a Ramsey class by Theorem~\ref{thm:models2}. It is perhaps surprising that within this class one can axiomatise $\vec {\mathcal S}_{k,t}$.
 Consider the  following class $\K$ of structures
$\str{A} = (A,\leq_\str{A},\func{A}{})$  where:

\begin{enumerate}
 \item $A$ is a linearly ordered set by $\leq_\str{A}$;
 \item $\func{A}{}$ is a partial function from $A^t$ to  $A\choose k$;
  \item \label{cond:hvezdicka}every $\vec{t}\in \dom(\func{A}{})$ has no repeated vertices;
 \item \label{cond:hvezdicka2}for every $\vec{t}\in \dom(\func{A}{})$ it holds that every vertex of $\vec{t}$ is in $\func{A}{}(\vec{t})$ and
every $r$-tuple $\vec{t}_2$ of distinct vertices of $\func{A}{}(\vec{t})$ is in $\dom(\func{A}{})$ and $\func{A}{}({\vec{t}})=\func{A}{}(\vec{t}_2)$.
\end{enumerate} 
The embeddings and subobjects for structures in $\mathcal{K}$ are defined as in Section~\ref{intro}.

Let us formulate explicitely what is the meaning of embeddings in this case:
According to the above definitions a monotone injection $f: A \to B$ is an embedding if and only if $\func{B}{}((f(x_1),f(x_2),\allowbreak \ldots,\allowbreak f(x_t))) = f(\func{A}{}(x_1,x_2,\ldots, x_t))$  whenever one of the sides of this equation makes sense. This then translates to the following notion of strong embedding:
For partial linearly ordered $(k,t)$-Steiner systems $G = (V,E)$ and $G' = (V',E')$ a monotone injective mapping $h: V \to V'$ is a \emph{strong embedding} if $h$ is a monotone embedding (as in relational structures, i.e. in this case hypergraphs) of $G$ into $G'$ and with the additional property that no $t$-tuple  
of $V$ belongs to a $k$-tuple of $G'$ outside of $G$.
Strong embeddings and strong subobjects correspond exactly to embeddings and (closed) substructures of the above structures $\str{A} = (A,\leq_\str{A},\func{A}{})$. Denote (for a second) by $\mathcal R$ the class of all such finite structures $\str{A}$ where $\leq_\str{A}$ is a linear ordering of vertices. 

It is clear that some structures of $\mathcal R$ correspond to partial $(k,t)$-Steiner systems
with a linear ordering of its vertices. The function $F$ assigns to a given $t$-tuple of distinct elements the unique set of $k$ elements containing it, providing such a set exists. Thus the range of the function $F$ is the set of edges of a partial $(k,t)$-Steiner system.
However not all structures in $\mathcal{R}$ correspond to $\vec{\mathcal{S}}_{k,t}$ because they may not satisfy conditions~\ref{cond:hvezdicka} and~\ref{cond:hvezdicka2} above.
Conditions~\ref{cond:hvezdicka} and~\ref{cond:hvezdicka2} are  clearly local
and
induces a subclass $\mathcal{K}$
of $\mathcal{R}$ which can be clearly identified with  $\vec{\mathcal{S}}_{k,t}$.
Consequently, by application of Theorem \ref{thm:mainstrongclosures} we get:


\begin{thm}
The class $\mathcal{K}$ (and consequently  $\vec{\mathcal{S}}_{k,t}$ with strong embeddings)
is a Ramsey class.
\end{thm}
\begin{proof}
In analogy to the proof of Theorem~\ref{equivalence} we can show that $\mathcal K$
is an $\mathcal R$-multiamalgamation class. The completion property follows similarly, too:
Given $\str{A},\str{B}\in \mathcal K$ and $\str{C}_0\in \mathcal R$ put $n(\str{B},\str{C}_0)=0$.
For $\str{C}\longrightarrow (\str{B})^\str{A}_2$ with a homomorphism-embedding to $\str{C}_0$
construct $\str{C}'$ on vertex set $C$ by completing $\leq_{\str{C}}$ to a linear
order $\leq_{\str{C}'}$ and putting $\func{C'}{}(\vec{t})=E$ if and only if $\func{C}{}(\vec t)=E$ and  $\vec{t}\cup E\subseteq \widetilde{B}$ for some $\widetilde{\str{B}}\in {\str{C}\choose \str{B}}$.
\end{proof}

This Theorem was proved in  \cite{bhat2016ramsey} by an explicit construction.
A different proof (not using Theorem \ref{thm:mainstrongclosures}) is given in \cite{Evans3}.

Note also that  Steiner systems (such as Steiner Triple Systems) correspond to
those structures $\str{A} = (A,\leq_\str{A},\func{A}{})\in \K$ where $\func{A}{}$ is a total (not partial) function and, equivalently, to the irreducible structures in $\mathcal{K}$.
It is a nontrivial fact that non-trivial Steiner systems exist for all values
$(k,t)$ and that any partial  $(k,t)$-Steiner systems can be extended to a (full)
Steiner system, see \cite{Keevash2014,wilsonI,wilsonII,wilsonIII}.
But this fact together with Theorem \ref{equivalence} implies that ordered $(k,t)$-Steiner systems form a Ramsey class.
Note also that the first non-trivial Ramsey property for Steiner systems was 
proved in \cite{Nevsetvril1987}.

\section{Degenerate Graphs and Graphs with $k$-ori\-en\-ta\-tions (two views)}
\label{sec:orientedgraphs}
Here we give yet another interpretation of Ramsey results for finite models.
Given a positive integer $k$, a {\em $k$-orientation} is an oriented graph such that the
out-degree of
every vertex is at most $k$.  (Thus a particular case of such graphs are degenerate graphs or graphs with maximal average degree bounded by $k$; see \cite{Nevsetvril2012} for a comprehensive study of such ``sparse'' classes.)
A $k$-orientation $G_1=(V_1,E_1)$ is {\em successor closed} in a
$k$-orientation $G_2=(V_2,E_2)$ if
$G_1$ is a subgraph of $G_2$ and moreover there is no edge oriented from
$V_1$ to $V_2\setminus V_1$.

Denote by $\mathcal D_k$ the class of all finite $k$-orientations.
This is
a hereditary class closed for free amalgamation over
successor-closed subgraphs
and thus the successor-closeness plays the role of embeddings and strong
substructures.

We consider a language $L=\{\func{}{1},\func{}{2},\ldots,\func{}{k}\}$ consisting
of function $\func{}{i}$ of domain arity 1 and range arity $i$, for $1\leq i\leq k$.
Given an oriented graph $G=(V,E)\in \mathcal D_k$ denote by
$\str{G}^+$ the structure with vertex set
$V$ and unary functions $\func{}{1},\func{}{2},\ldots, \func{}{k}$.
Function $\func{}{i}$, $1\leq i\leq k$,
is defined for every vertex of outdegree $i$ and maps the vertex to
all vertices in its out-neighbourhood.
Denote by $\mathcal D^+_k$ the class of all structures $\str{G}^+$
for $G\in \mathcal D_k$.
Because $\str{G}^+_1$ is a substructure of $\str{G}^+_2$ if and
only if $G_1$ is successor
closed in $G_2$ it follows that $\mathcal D^+_k$ is a free
amalgamation class. 

Denote by $\overrightarrow{\mathcal D}^+_k$ a class
of all structures $\overrightarrow{\str{A}}=(A,\leq_\str{A},\func{A}{1},\func{A}{2},\ldots, \func{A}{k})$
such that $\str{A}=(A,\leq_\str{A},\func{A}{1},\func{A}{2},\ldots, \func{A}{k})\in \mathcal D_k$ and
$\leq_\str{A}$ is a linear order of $A$.
We immediately get:
\begin{thm}
Class
$\overrightarrow{\mathcal D}^+_k$ is a Ramsey class.
\end{thm}
This can be seen as the most elementary use of
Theorems~\ref{thm:mainstrongclosures}. But there is more to this than meets the eye: The classes $\mathcal D_k$ and $\overrightarrow{\mathcal D}^+_k$ are building blocks for the analysis of the Hrushovski construction in the context of structural Ramsey theory \cite{Evans2}. At the end these examples 
provide the
first example of an $\omega$-categorical structure without a precompact Ramsey lift \cite{Evans2}.

It is interesting that the class $\mathcal D_k$ allows a different  lift (or expansion) which makes it also a Ramsey class. 
Let $G = (V,E)$ be a $k$-oriented graph.  For every  vertex $v$ let us fix some
arbitrary enumeration of arcs starting from $v$: $e_1(v),e_2(v),\ldots,e_{k'}(v)$ for $k' \leq k$.
The the set $E$ is decomposed into sets $E_1, E_2,\ldots, E_k$ where each set $E_i$ is an oriented graph with outdegrees $\leq 1$ and thus each $E_i$  can be viewed as a  mapping $F_i$ (by adding idempotent elements at vertices where there are no outgoing arcs).

Thus we get the  structures with language $L = \{\func{}{1},\func{}{2},\ldots,\func{}{k}\}$ and $L$-models $\str{A}$ representing $(A,\func{A}{1},\func{A}{2},\ldots,\func{A}{k})$ where $\func{A}{i}$ are functions $A \to A$.
In this situation we can apply Theorem \ref{thm:models2} directly. The additional condition is that 
the functions $A \to A$
 have (except for idempotents) disjoint sets of arcs. However this is clearly a local condition.
  Denote by $\mathcal{F}_L$ the class of all such $L$-models together with embeddings (defined in Section \ref{intro}). It is clear that the class $\mathcal{F}_L$ is locally finite and that it has free amalgamation. Again denote by $\overrightarrow{\mathcal F}_L$ the class of all structures $\overrightarrow{\str{A}}$ extending structures $\str{A}\in \mathcal{F}_L$ by binary relation $\leq_\str{A}$ defining a linear ordering of vertices.
Thus we have:

\begin{thm}
\label{korient}
The class $\overrightarrow{\mathcal{F}}_L$ is a Ramsey class.
In the other words: The class of all ordered oriented graphs which are decomposable into
$k$
partial  mappings is a Ramsey class. 
\end{thm} 

By different more elaborate methods this has been proved in \cite{Sokic2016}.
An even simpler direct proof is given in \cite{Hubicka2016}.
Note that the closed sets in classes $\overrightarrow{\mathcal{F}}_L$ and $\overrightarrow{\mathcal D}^+_k$ (which are ``expansions'' of the same class $\mathcal D_k$) are in  a correspondence. However the objects differ in their ``labelling''.
The formulation of a Ramsey expansion  by means of decomposition will be further pursued in the next section.

\section{Resolvable Block Designs}
\label{resolvable}

We can combine the  results of previous two sections. Here is a particular combinatorially interesting example (with illustrious history \cite{wilsonI,wilsonII,wilsonIII,ray1971,NN1971}).

Let $k > 2$. A \emph{Balanced Incomplete Block Design} (shortly BIBD) with parameters $(v,k,1)$ is a $k$-uniform hypergraph $(X,\mathcal{B})$ where $\mathcal{B}$ is a system of $k$-element subsets of $X, |X| = v,$ such that every two element subset of $X$ is contained in exactly one  subset (called \emph{block}) in $\mathcal{B}$. 
(In fact this is the same definition as Steiner systems: our BIBD are $(v,k,1)$-BIBD. We have chosen this term here as it is more standard in this area of resolvability. For $\lambda > 1$ it is possible to generalise our argument. For the sake of brevity we do not do so in this paper.)
A partial incomplete balanced block design (PBIBD) contains every pair of elements of $X$ in at most one block. A  {\em Resolvable Balanced Incomplete Block Design} (shortly {\em RBIBD}) is block design with the additional property that $\mathcal{B}$ is partitioned into equivalences on $X$ with block as  equivalence classes (thus $|X|$ has to be divisible by $k$). In this case we denote blocks of these equivalences as 
$\mathcal{B}_1,\mathcal{B}_2,\ldots,\mathcal{B}_m$ (where of course $t = (|X|-1)/(k-1))$. A \emph{partial} RBIBD (shortly PRBIBD) is just a partition of blocks into equivalences (not necessarily covering the whole $X$).

If $(X,\mathcal{B}) = (X,\mathcal{B}_1,\mathcal{B}_2,\ldots,\mathcal{B}_m)$ and $(X', \mathcal{B'}) = (X',\mathcal{B}_1,\mathcal{B}'_2,\ldots,\mathcal{B}'_{m'})$  are two (partial) RBIBD
then isomorphism, embedding, substructure are defined as strong mappings (Section~\ref{sec:steiner}) and moreover preserving equivalence (i.e. partition into sets $\mathcal{B}_1,\mathcal{B}_2,\ldots,\mathcal{B}_m$).

Denote by $\mathcal{PRBIBD}$ the class of all  partial resolvable BIBD and all their embeddings and
denote by $\overrightarrow{\mathcal{PRBIBD}}$ the class of all linearly ordered PRBIBD and all their embeddings.
One can prove that $\overrightarrow{\mathcal{PRBIBD}}$ is a $\mathcal R$-multiamalgamation class (such complicated abbreviations seem to by typical for this area). In fact this is a combination of  Sections \ref{sec:steiner} and Section \ref{sec:orientedgraphs}.  This can be outlined as follows:

The class $\overrightarrow{\mathcal{PRBIBD}}$ may be interpreted in our scheme (in order to apply Theorem \ref{thm:mainstrongclosures}) as follows.
The language is $L = \{\leq,S, \func{}{1},\func{}{2}\}$  with arities
$d(\func{}{1}) = 2, r(\func{}{1}) =k, d(\func{}{2}) = k,r(\func{}{2}) = 1, a(R^S) = 1$. We consider the following class $\mathcal R$ of $L$-structures 
$\str{A} = (A,\leq_\str{A},S_\str{A},\func{A}{1},\func{A}{2})$:

\begin{enumerate}
 
 \item $\func{A}{1}$ is a partial function from $A\choose 2$ to  $A\choose k$ (function $\func{A}{1}$ will represent the blocks of $\str A$);
 \item $\func{A}{2}$ is a partial function from $A\choose k$ to  $A$ (function $\func{A}{2}$ will represent the equivalences).
\end{enumerate}

(We want the functions $\func{A}{1}$ and $\func{A}{2}$ to be on tuples of distinct elements and moreover symmetric. We indicated this by choosing the domains  $A\choose 2$ and   $A\choose k$.)
The embeddings are defined in the usual way. Essentially these are strong maps (see Section \ref{sec:steiner}) which preserve partition of blocks into equivalences.)
Clearly, again by Theorem \ref{thm:models2}, $\mathcal{R}$ is a Ramsey class.

However our class $\overrightarrow{\mathcal{PRBIBD}}$  is the class of structures which satisfies a few more axioms which relate relations and functions together. They can be 
formulated as follows (for brevity we put $S = \{x, (x) \in S\}$:

\begin{enumerate}
\item $\leq_\str{A}$ is a linear ordering of $A$;
\item for every $v\in \vec{t}\in \dom(\func{A}{2})$ it holds that $v\notin S$;
\item $\left (\bigcup \mathrm{Rg}(\func{A}{2})\right) \subseteq S$;
  \item for every $\{x,y\} \in \dom(\func{A}{1})$ it holds that both $x$ and $y$ are in $\func{A}{1}(\vec{t})$;
\item for every $\{x,y\} \in \dom(\func{A}{1})$ it holds that
every pair $x',y'$ of distinct vertices of $\func{A}{1}(x,y)$ is in $\dom(\func{A}{1})$ and $\func{A}{1}(x',y')=\func{A}{1}(x,y)$;
\item the domain of $\func{A}{2}$ is equal to the range of $\func{A}{1}$ (i.e. $\dom(\func{A}{2})$ are all blocks);
\item if $\func{A}{2}(K) =\func{A}{2}(K')$ then $K$  and $K'$  are disjoint sets.
\end{enumerate}

Denote by $\mathcal{K}$ the subclass of $\mathcal{R}$ of all $L$-structures with these properties. Observe that these are all local conditions and
one can prove that with this   the class $\mathcal{K}$ is an $\mathcal{R}$-multiamalgamation class and thus 
Theorem \ref{thm:mainstrongclosures} can be applied. Moreover one sees easily that this axiomatisation describes the class $\overrightarrow{\mathcal{PRBIBD}}$.
As a consequence we have the following 

\begin{thm}
\label{RSD}
The classes $\mathcal{K}$ and $\overrightarrow{\mathcal{PRBIBD}}$ are Ramsey classes.
\end{thm}

The previous construction may be further exploited to other partitioned and, say, $\mathcal H$ decomposable structures.
In the next section we give such an example. However we do not aim for generality here.

\section{$H$-factorization theorem}
In this example we consider undirected graphs with embeddings (i.e. induced subgraphs). Suppose that $H$ is a fixed graph (see remark at the end concerning possible generalisations).
Recall that $G\choose H$ the set of all induced subgraphs of $G$  isomorphic to  $H$.
An \emph{$H$-matching} in a graph $G$ is a collection $H_1$, $H_2$, \ldots, $H_t$ (vertex disjoint) induced subgraphs of $G$ (i.e. we assume that
$V(H_i)\cap V(H_j)=\emptyset$ for $i\neq j$) such that all $H_i$ are isomorphic to $H$. An $H$-matching will be denoted by $M_H$, thus in our case $M_H=(H_1,H_2,\ldots,H_m)$. An \emph{$H$-factorization} of $G$ is a collection of $H$-matchings which partition  the set $G\choose H$. \emph{Partial $H$-factorization} of $G$ is a collection of disjoint $H$-matching in the set $G\choose H$.

Graphs with a partial $H$-factorization will be considered with an embedding preserving the factorization. Given $G=(V,E)$ and $G'=(V',E')$ with partial factorizations
$(M^1_H(G),M^2_H(G),\ldots,M^m_H(G))$ and $(M^1_H(G'),M^2_H(G'),\ldots,\allowbreak M^{m'}_H(G'))$ we say that  a pair $f , \iota$  
is an  $H$-factorization embedding if:
\begin{enumerate}
 \item $f$ is an embedding $G$ into $G'$,
\item $\iota$ is a monotone injection $\{1,\ldots, m\} \to \{1,\ldots,m'\}$,
 \item for every two graphs $H', H'' $ isomorphic to $H$ we have $H',H'' \in M^i_H(G)$ if and only if $f(H'),f(H'' )\in M^{\iota(i)}_H(G)$.
\end{enumerate}
The last condition of course means that two graphs $H'$, $H''$ isomorphic to  $H$ belong to the same class of partial $H$-factorizations of $G$ if and only if their images $f(H'), f(H'')$ belong to the same class of partial $H$-factorizations of $G'$.

Denote by $\mathrm{Fact}(H)$ the class of all finite graphs endowed with partial $H$-factorizations. By $\overrightarrow{\mathrm{Fact}}(H)$ we denote the corresponding class of linearly ordered partially $H$-factorized graphs endowed with monotone embedding. We shall see that   $\overrightarrow{\mathrm{Fact}}(H)$ can be always interpreted as an multi-amalgamation class. 
 For this it is convenient to consider structures in $\mathrm{Fact}(H)$ as structures with  two mappings, and two relations as follows:

Again, we consider the language $L = \{R,\leq,R^S,\func{}{1},\func{}{2}\}$.
 Function $\func{}{1}$ has domain arity $2$ and range arity $|H|$. The function 
$\func{}{2}$  has  domain of arity $|H|$ and range of arity $1$, $R^S$ has arity $1$.

In this situation we consider the class $\mathcal{R}$ of all $L$-structures 
$\str{A} = (A,\rel{A}{},\leq_{\str A}\nobreak ,\allowbreak R^S_{\str{A}},\allowbreak \func{A}{1},\func{A}{2})$:
Again applying Theorem \ref{thm:models2} we know that $\mathcal R$ is a Ramsey class.
However the class $\mathrm{Fact}(H)$ has some more properties and they will be reflected by the  subclass $\mathcal K$ of $\mathcal R$ of all structures
$\str{A} = (A,\rel{A}{},\leq_{\str{A}}\nobreak ,\allowbreak S_\str{A},\func{A}{1},\func{A}{2})$ which satisfy the following
(for brevity, we put again $S = \{x, (x) \in R\}$):

\begin{enumerate}
\item $\leq_\str{A}$ is a linear ordering of $A$;
\item for every $v\in \vec{t}\in \dom(\func{A}{2})$ it holds that $v\notin S$;
\item $\left (\bigcup \mathrm{Rg}(\func{A}{2})\right) \subseteq S$;
\item  $(A,\rel{A}{})$ is an undirected graph (denoted by $G_{\str{A}})$;
\item for every $t$-tuple $\vec{t}\in \dom(\func{A}{2})$  the graph  $G_{\str{A}}$ induces on $\vec{t}$ a graph isomorphic to  ${H}$;
 \item $\func{A}{1}$ is a partial function from $A^2$ to  $A^{|H|}$;
\item  $\func{A}{2}$ is a partial function from $A^{|H|}$ to $A$;
\item $\func{A}{1}$ is symmetric; explicitely $\func{A}{1}(u,v) = \func{A}{1}(v,u)$;
\item $\func{A}{2}$ is symmetric and defined on tuples of distinct elements;
  \item for every pair $u,v\in \dom(\func{A}{1})$
 it holds that both vertices $u,v$  belong also to the image  $\func{A}{1}(u,v)$;
\item every pair $u',v'$  of distinct vertices such that $\func{A}{1}(u,v)=\{u',v'\}$ is in $\dom(\func{A}{1})$ and $\func{A}{1}(u,v)=\func{A}{1}(u',v')$;
\item the domain of $\func{A}{2}$ is equal to the range of $\func{A}{1}$;
\item $\func{A}{2}$ is ``idempotent'' meaning that $\dom(\func{A}{2}) \cap \mathrm{Rg}(\func{A}{2}) = \emptyset$;
\item if $\func{A}{2}(K) =\func{A}{2}(K')$ then $K$  and $K'$  are disjoint sets.
\end{enumerate}

All the conditions on functions $\func{A}{1}$ and $\func{A}{2}$ are local and 
one can prove that with this interpretation the class $\mathcal{K}$ is a $\mathcal R$-multiamalgamation class and thus 
Theorem \ref{thm:mainstrongclosures} can be applied.
As a consequence we have the following 

\begin{thm}
Then $\mathcal K$ is a Ramsey class.
\end{thm}

\section {Final comments}
\paragraph{1.} The key constructions of this paper may be iterated. In this paper we presented iterations of at most two functions.
Examples given here illustrate two main features of applications (and influence) of inclusion of functions in our vocabularies: we have trees and equivalences, functions and partitions.

\paragraph{2.} The calculus of functions and partitions leads to abstract data types and Ramsey classes which may have some computer science consequences. There are some precedents, see influential \cite{yao1981}. 

\paragraph{3.}
We may have more arities. For example we may consider $\mathcal H$-factorisation for a family $\mathcal H$ of graphs. The arities are then  the set $\{|H|, H \in \mathcal{H}\}$. This simply means  that the domain and range can contain any arity in the set. We may even have  
an infinite set of arities. In dealing with Ramsey classes this does not lead to difficulties as we always deal with finite structures of sizes given in advance. (Recall the basic scheme: Given $\str{A, B}$, we look for $\str{C}$.) One can clearly generalise here.

We decided to stop here.

\bibliography{ramsey.bib}
\end{document}